\documentclass[12pt,final]{article}
\usepackage[margin=2.5cm,top=1.5cm]{geometry}
\usepackage[all,arc]{xy}
\usepackage{mathrsfs}
\usepackage{tikz}
\usepackage{graphicx}
\usepackage{mathtools}
 \usepackage{subfigure}
\usepackage[centerlast]{caption2}
\usepackage{amsmath,amsthm,amssymb,enumerate}
\usepackage{color}
\usepackage{setspace}

\newtheorem{defn}{Definition}[section]
\newtheorem{lem}[defn]{Lemma}
\newtheorem{theo}[defn]{Theorem}
\newtheorem{cor}[defn]{Corollary}

\newtheorem{prop}[defn]{Proposition}

\newtheorem{exam}[defn]{Example}
\newtheorem{con}[defn]{Conjecture}
\newtheorem{prob}[defn]{Problem}

\newtheorem{rem}[defn]{Remark}

\newcommand{\Romannum}[1]{\uppercase\expandafter{\romannumeral #1}}

\numberwithin{equation}{section}

\newcommand\keywordsname{Key words}
\newcommand\AMSname{AMS subject classifications}

\newenvironment{@abssec}[1]{%
     \if@twocolumn
       \section*{#1}%
     \else
       \vspace{.05in}\footnotesize
       \parindent .2in
         {\upshape\bfseries #1. }\ignorespaces
     \fi}
     {\if@twocolumn\else\par\vspace{.1in}\fi}

\begin{document}

\vskip6cm
\title{Spectral radius and signless Laplacian spectral radius of strongly connected digraphs
\footnote{Research supported 
the Zhujiang Technology New Star Foundation of
Guangzhou (No. 2011J2200090) and Program on International Cooperation and Innovation, 
Department of Education, Guangdong Province (No. 2012gjhz0007).}}
\author{Wenxi Hong, Lihua You\footnote{{\it{Corresponding author:\;}}ylhua@scnu.edu.cn.}}
\vskip.2cm
\date{{\small
School of Mathematical Sciences, South China Normal University,\\
Guangzhou, 510631, P.R. China\\
}} \maketitle

\begin{abstract}

 \vskip.3cm
Let $D$ be a strongly connected digraph and $A(D)$ be the adjacency matrix of $D$. Let $diag(D)$ be the diagonal matrix with outdegrees of the vertices of $D$ and $Q(D)=diag(D)+A(D)$ be the signless Laplacian matrix of $D$. The spectral radius of $Q(D)$ is called the signless Laplacian spectral radius of $D$, denoted by $q(D)$. In this paper, we  give sharp bound on $q(D)$  with outdegree sequence and compare the bound with some known bounds, establish some sharp upper or lower bound on $q(D)$ with some given parameter such as clique number, girth or vertex connectivity, and characterize the corresponding extremal digraph or proposed open problem.  In addition, we also determine the unique digraph which achieves the minimum (or maximum), the second minimum (or maximum), the third minimum, the fourth minimum spectral radius and signless Laplacian spectral radius among all strongly connected digraphs, and answer the open problem proposed by Lin-Shu [H.Q. Lin, J.L. Shu, A note on the spectral characterization of strongly connected bicyclic digraphs, Linear Algebra Appl. 436 (2012) 2524--2530].
\vskip.2cm \noindent{\it{AMS classification:}} 05C20; 05C50; 15A18
 \vskip.2cm \noindent{\it{Keywords:}} Digraph; Signless Laplacian; Spectral radius; Clique number; Girth; Vertex connectivity.
\end{abstract}

\section{ Introduction}
\hskip.6cm
Let $D=(V(D), E(D))$ be a digraph, where $V(D)=\{1, 2, \ldots, n \}$ and $E(D)$ are the  vertex set and arc set of $D$, respectively.
A digraph $D$ is simple if it has no loops and multiple arcs.
A digraph  $D$ is strongly connected if for every pair of vertices $i, j\in V(D)$, there are a directed path from $i$ to $j$ and a directed path from $j$ to $i$.
In this paper, we consider finite, simple strongly connected digraphs, simply, strongly connected digraphs. We follow \cite{2001, 1979, 1976} for terminology and notations.

Let $D$ be a digraph. If two vertices are connected by an arc, then they are called adjacent.
For $e=(i, j)\in E(D)$, $i$ is the initial vertex of $e$, $j$ is the terminal vertex of $e$ and vertex $i$ is the tail of vertex $j$.
Let $N^-_D(i)=\{j\in V(D)|(j, i)\in E(D)\}$ and $N^+_D(i)=\{j\in V(D)|(i, j)\in E(D)\}$ denote the in-neighbors and out-neighbors of $i$, respectively.
Let $d^-_i=|N^-_D(i)|$ denote the indegree of the vertex $i$ and $d^+_i=|N^+_D(i)|$ denote the outdegree of the vertex $i$ in $D$.
If $d^+_1=d^+_2=\cdots=d^+_n$, then $D$ is a regular digraph. Let $t^+_i=\sum\limits_{j\in N^+_D(i)} d^+_j$  be the 2--outdegree of the vertex $i$,  $m^+_i=\frac{t^+_i}{d^+_i}$ be the average 2--outdegree of the vertex $i$.

 Let $\overrightarrow{P_n}$ and $\overrightarrow{C_n}$ denote the directed path and the directed cycle on $n$ vertices, respectively. Let $\overset{\longleftrightarrow}{K_n}$ denote the complete digraph on $n$ vertices in which two arbitrary vertices $i, j\in V(\overset{\longleftrightarrow}{K_n})$, there are arcs $(i, j), (j, i)\in E(\overset{\longleftrightarrow}{K_n})$.

 Let $F$ be a subdigraph of $D$. If $D[V(F)]$ is a complete subdigraph of $D$, then $F$ is called a clique of $D$.
 The clique number of a digraph $D$, denoted by $\omega(D)$, is the maximum value of the numbers of the vertices of the  cliques in $D$.
 The girth of $D$ is the length of the shortest directed cycle of $D$.
 The vertex connectivity of $D$, denoted by $\kappa(D)$, is the minimum number of vertices whose removal destroys the strong connectivity of $D$.

Let $M$ be an $n\times n$ nonnegative matrix, $\lambda_1, \lambda_2, \ldots, \lambda_n$ be the eigenvalues of $M$.
 It is obvious that the eigenvalues can be complex numbers since $M$ is not symmetric in general.
 We usually assume that $|\lambda_1|\geq |\lambda_2|\geq\ldots \geq |\lambda_n|$.
 The spectral radius of $M$ is defined as $\rho(M)=|\lambda_1|$, i.e., it is the largest modulus 
 of the eigenvalues of $M$.   Since $M$ is a nonnegative matrix, it follows from the Perron Frobenius Theorem that $\rho(M)$ is a real number.

 For a digraph $D$, let $A(D)=(a_{ij})$ denote the adjacency matrix of $D$, where $a_{ij}$ is equal to the number of arc $(i, j)$.
 The spectral radius of $A(D)$, denoted by $\rho(D)$, is called the spectral radius of $D$.

 Let $diag(D)=diag(d^+_1, d^+_2, \ldots, d^+_n)$ be the diagonal matrix with outdegree of the vertices of  $D$ and
 $Q(D)=diag(D)+A(D)=(q_{ij})$ be the signless Laplacian matrix of $D$.
 The spectral radius of $Q(D)$, $\rho(Q(D))$, denoted by $q(D)$, is called the signless Laplacian spectral radius of $D$.

Since $D$ is a strongly connected digraph, then $A(D)$, $Q(D)$ are nonnegative irreducible matrices. It follows from the Perron Frobenius Theorem that
  $\rho(D)$ and $\rho(Q(D))=q(D)$ are positive real numbers and there is a positive unit eigenvector corresponding to $\rho(D)$ and  $q(D)$, respectively.

 The  spectral radius and the signless Laplacian spectral radius  of undirected graph  are well treated in the literature,
 see \cite{2009PIM, 2010LAA, 2013, 2010LAA1,  2013LAA, 2013LAA2, 2014LAA} and so on,
 but there is not much known about digraphs.
Recently, R.A. Brualdi  wrote a stimulating survey on the spectra of digraphs\cite{2010}. Furthermore,
some upper or lower bounds on  the  spectral radius or the signless Laplacian spectral radius  of digraphs were obtained
by the outdegrees and the average 2-outdegrees \cite{2013ARS, 2013}.
Some  extremal digraphs which attain the maximum or minimum spectral radius 
and the distance  spectral radius  of digraphs with given parameters, such as given connectivity, given arc connectivity,
given dichromatic number, given clique number, given girth  and so on, were characterized,
see e.g. \cite{2013DM, 2011LAA,  2013DAM, 2012DM, 2012DM2}.

 In \cite{2013ARS}, S. Burcu Bozkurt and Durmus Bozkurt gave some sharp upper and lower bounds for the signless Laplacian spectral radius as follows:

\begin{equation}\label{eq11}
\min\{d^+_i+d^+_j:(i, j)\in E(D)\}\leq q(D) \leq\max\{d^+_i+d^+_j:(i, j)\in E(D)\},
\end{equation}
\begin{equation}\label{eq12}
\min\{d^+_i+m^+_i:i\in V(D)\}\leq q(D) \leq\max\{d^+_i+m^+_i:i\in V(D)\},
\end{equation}
\begin{equation}\label{eq13}
q(D)\leq\max\{\frac{d^+_i+d^+_j+\sqrt{(d^+_i-d^+_j)^2+4m^+_im^+_j}}{2}:(i, j)\in E(D)\},
\end{equation}
\begin{equation}\label{eq14}
q(D)\leq\max\{d^+_i+\sqrt{t_i^+}:i\in V(D)\}.
\end{equation}

The paper is organized as follows: we give a new sharp upper bound of the signless Laplacian spectral radius with outdegree sequence among strongly connected digraphs and compare the bound with some known bounds in  Section 2;
 In Section 3, 
 we give some graph transformations on digraphs which are useful to the proofs of our main results;
 In Section 4, we characterize the digraph which minimizes the signless Laplacian spectral radius with given clique number;
 In Section 5, we characterize the digraph which minimizes the signless Laplacian spectral radius with given girth
 and determine the unique digraph with the second, the third and the fourth  minimum signless Laplacian spectral radius among all strongly connected digraphs;
In Section 6, we study the  maximum signless Laplacian spectral radius with given vertex connectivity among all strongly connected digraphs,
describe the extremal digraph set and the value of the signless Laplacian spectral radius of such digraphs,
determine the unique digraph with the second maximum signless Laplacian spectral radius among all strongly connected digraphs,
and  conjecture the digraph which maximizes the signless Laplacian spectral radius with given vertex connectivity.
In Section 7, we determine the unique digraph which achieves  the second minimum (or maximum), the third minimum, the fourth minimum spectral radius among all strongly connected digraphs and answer the open problem proposed by Lin-Shu [H.Q. Lin, J.L. Shu, A note on the spectral characterization of strongly connected bicyclic digraphs, Linear Algebra Appl. 436 (2012) 2524--2530].

\section{Bounds on the signless spectral radius of digraphs}

\hskip.6cm
In this section, we first give the maximal and minimal signless spectral radius of digraphs. Then we give a new sharp upper bound of the signless Laplacian spectral radius among all simple digraphs and compare them with the upper bounds given in \cite{2013ARS} as inequalities (1.1)--(1.4).
 The technique used in the result is motivated by  \cite{2013, 2013LAA2} et al.

\begin{lem}\label{lem21}{\rm(\cite{1979})}
If $A$ is an $n\times n$ nonnegative matrix with the spectral radius $\rho(A)$ and row sums $r_1, r_2, \ldots, r_n$, then
$\min\limits_{1\leq i\leq n}r_i\leq \rho(A)\leq \max\limits_{1\leq i\leq n}r_i$. Moreover, if $A$ is irreducible, then one of the equalities holds if and only if the row sums of $A$ are all equal.
\end{lem}

By the definition of $Q(D)$, the $i$-th row sum of $Q(D)$ is $2d_i^+$. Follows from Lemma \ref{lem21}, we can get $q(\overrightarrow{C_n})=2$ and $q(\overset{\longleftrightarrow}{K_n})=2n-2$ immediately.

\begin{defn}\label{defn22}{\rm(\cite{1979}, Chapter 2)}
Let $A=(a_{ij}), B=(b_{ij})$ be $n\times n$ matrices. If $a_{ij}\leq b_{ij}$ for all $i$ and $j$, then $A\leq B$.
If $A\leq B$ and $A\neq B$, then $A< B$. If $a_{ij} < b_{ij}$ for all $i$ and $j$, then $A \ll B$.
\end{defn}

\begin{lem}\label{lem23}{\rm(\cite{1979}, Chapter 2)}
Let $A, B$ be $n\times n$ matrices with the spectral radius $\rho(A)$ and $\rho(B)$. If $0\leq A\leq B$, then $\rho(A) \leq \rho(B)$.
Furthermore, if $0\leq A<B$ and $B$ is irreducible, then $\rho(A) < \rho(B)$.
\end{lem}

\begin{lem}\label{lem24}{\rm(\cite{1979}, Chapter 2; \cite{1988}, Chapter 1)}
Let $m< n$, $A, B$ be $n\times n$, $m\times m$ nonnegative matrices with the spectral radius $\rho(A)$ and $\rho(B)$, respectively.
If $B$ is a principal submatrix of $A$, then $\rho(B) \leq \rho(A)$.
Furthermore, if $A$ is irreducible, then $\rho(B) < \rho(A)$.
\end{lem}

By Lemmas \ref{lem23}--\ref{lem24} and the definitions of $Q(D)$ and $q(D)$, we have the following  results in terms of digraphs.

\begin{cor}\label{cor25}
Let $D$ be a digraph and $H$ be a subdigraph of $D$. Then $q(H)\leq q(D)$. If $D$ is strongly connected, and $H$ is a proper  subdigraph of $D$, then $q(H)<q(D)$.
\end{cor}

From Lemma \ref{lem21} and Corollary \ref{cor25}, we easily get the following results.

\begin{cor}\label{cor26}
Let $D$ be a strongly connected digraph. Then $2\leq q(D)\leq 2(n-1)$,
$q(D)=2n-2$ if and only if $D\cong \overset{\longleftrightarrow}{K_n}$,
and $q(D)= 2$  if and only if $D\cong \overrightarrow{C_n}$.
\end{cor}

\begin{theo}\label{theo27}
Let $D=(V(D),E(D))$ be a simple  digraph on $n\geq 2$ vertices with $V(D)=\{1,2,\ldots, n\}$, outdegree sequence $d^+_1, d^+_2, \ldots, d^+_n$, where $d^+_1\geq d^+_2\geq\cdots\geq d^+_n$.
Let $\phi_1=2d_1^+$ and for $2\leq l\leq n$,
\begin{equation}\label{eq21}
\phi_l=\frac{d^+_1+2d^+_l-1+\sqrt{(2d^+_l-d^+_1+1)^2+8\sum\limits_{i=1}^{l-1}(d^+_i-d^+_l)}}{2}.
\end{equation}
\noindent  and $\phi_s=\min\limits_{1\leq l\leq n}\{\phi_l\}$ for some $s\in \{1,2, \ldots, n\}$. Then $q(D)\leq \phi_s.$
Furthermore, if $D$ is a strongly connected digraph, then $q(D)=\phi_s$ if and only if $D$ is regular or
there exists an integer $t$ with $2\leq t\leq s$ such that $d^+_1=\cdots=d^+_{t-1}>d^+_t=\cdots=d^+_n$ and the indegrees $d^-_1=\cdots=d^-_{t-1}=n-1$.
\end{theo}
\begin{proof}
Firstly, we show $q(D)\leq \phi_l$ for all $1\leq l\leq n$.

{\bf Case 1: }  $l=1$.

It is obvious that $q(D)\leq \phi_1=2d_1$  by Lemma \ref{lem21} and the definition of $Q(D)$.

{\bf Case 2: }  $2\leq l\leq n$.

By (\ref{eq21}), it is obvious that $\phi_l\geq d_1^+-1$,  and $(\phi_l-2d^+_l)(\phi_l-d^+_1+1)=2\sum\limits_{i=1}^{l-1}(d^+_i-d^+_l)$.

 Let   $U=diag(x_1, \ldots, x_{l-1}, 1, \ldots,1)$ be an $n\times n$ diagonal matrix,
 where $x_i=1+\frac{2(d^+_i-d^+_l)}{\phi_l-d^+_1+1}$ for $i\in\{1, 2, \ldots, l-1\}$.
Then   
 $x_1\geq x_2\geq \dots\geq x_{l-1}\geq 1$, $\sum\limits_{k=1}^{l-1}(x_k-1)=\phi_l-2d_l^+$,
and $U^{-1}=diag(x^{-1}_1,  \ldots, x^{-1}_{l-1}, 1, \ldots,1)$.

Let $Q(D)=(q_{ij})_{n\times n}=diag(d_1^+, \ldots, d_n^+)+A(D)$ be the signless Laplacian matrix of $D$ and $B=U^{-1}Q(D)U$.
Obviously, $B$ and $Q(D)$ have the same eigenvalues, thus $q(D)=\rho(B)$. Let $r_i(B)$ $(i=1, 2, \ldots, n)$ be the row sums of $B$.
Now we show $r_i(B)\leq \phi_l$ for any $1\leq i\leq n$.

{\bf Subcase 2.1: }  $1\leq i\leq l-1$.

\hskip.6cm $r_i(B)=\sum\limits_{k=1}^{l-1}\frac{x_k}{x_i}q_{ik}+\sum\limits_{k=l}^{n}\frac{1}{x_i}q_{ik}$

\hskip1.65cm$=\frac{1}{x_i}\sum\limits_{k=1}^{l-1}x_kq_{ik}+\frac{1}{x_i}\sum\limits_{k=l}^{n}q_{ik}$

\hskip1.65cm$=\frac{2d_i^+}{x_i}+\frac{1}{x_i}\sum\limits_{k=1}^{l-1}(x_k-1)q_{ik}$

\hskip1.65cm $=\frac{2d_i^+}{x_i}+\frac{1}{x_i}\left((x_i-1)q_{ii}+\sum\limits_{k=1,k\neq i}^{l-1}(x_k-1)q_{ik}\right)$

\hskip1.65cm $\leq \frac{2d_i^+}{x_i}+\frac{1}{x_i}\left(d^+_1(x_i-1)+\sum\limits_{k=1,k\neq i}^{l-1}(x_k-1)\right)$

\hskip1.65cm $=\frac{1}{x_i}\left(2d_i^++(d^+_1-1)(x_i-1)+\sum\limits_{k=1}^{l-1}(x_k-1)\right)$

\hskip1.65cm $=\frac{1}{x_i}\left(2d^+_i+(d^+_1-1)\frac{2(d^+_i-d^+_l)}{\phi_l-d^+_1+1}+\frac{2\sum\limits_{k=1}^{l-1}(d^+_k-d^+_l)}{\phi_l-d^+_1+1}\right)$

\hskip1.65cm $=\frac{1}{x_i}\left(2d^+_i+(d^+_1-1)\frac{2(d^+_i-d^+_l)}{\phi_l-d^+_1+1}+\frac{(\phi_l-2d^+_l)(\phi_l-d^+_1+1)}{\phi_l-d^+_1+1}\right)$

\hskip1.65cm $=\phi_l$,

\noindent with equality if and only if (1) and (2) hold: (1) $x_i=1$ or $q_{ii}=d^+_1$ for $x_i>1$, (2) $x_k=1$ or $q_{ik}=1$ for $x_k>1$ if $1\leq k\leq l-1$ with $k\neq i$.

{\bf Subcase 2.2: } $l\leq i\leq n$.

\hskip.6cm
$r_i(B)=\sum\limits_{k=1}^{l-1}x_kq_{ik}+\sum\limits_{k=l}^{n}q_{ik}$
$=2d^+_i+\sum\limits_{k=1}^{l-1}(x_k-1)q_{ik}$
$\leq 2d^+_l+\sum\limits_{k=1}^{l-1}(x_k-1)$$=\phi_l$,


\noindent with equality if and only if (3) and (4) hold: (3) $d^+_i=d^+_l$, (4) $x_k=1$ or $q_{ik}=1$ for $x_k>1$
if $1\leq k\leq l-1$.


Hence by Lemma \ref{lem21},   $q(D)=\rho(B)\leq \max\limits_{1\leq i\leq n}\{r_i(B)\}\leq \phi_l$ for any $l\in \{2,3,\ldots, n\}$.
Thus  $q(D)=\rho(B)\leq \max\limits_{1\leq i\leq n}\{r_i(B)\}\leq \min\limits_{2\leq l\leq n}\{\phi_l\}$.
\vskip.1cm
Combining the above two cases, $q(D)\leq  \min\limits_{1\leq l\leq n}\{\phi_l\}$.

Let $D$ be a  strongly connected digraph, and $\phi_s=\min\limits_{1\leq l\leq n}\{\phi_l\}$ for some $s\in\{1, 2, \ldots, n\}$.

{\bf Case 1: }   $s=1$. 

It is obvious that $\rho(Q(D))=q(D)=\phi_1=2d_1^+$ if and only if $D$ is regular by Lemma \ref{lem21} and the fact $d^+_1\geq d^+_2\geq\cdots\geq d^+_n$.

{\bf Case 2: }  $2\leq s\leq n$. 

Clearly,  $Q(D)$ and $B$ are  irreducible nonnegative matrices because $D$ is a  strongly connected digraph.
Then $q(D)=\phi_s$ if and only if $\phi_1\geq \phi_s$, $\rho(B)=\max\limits_{1\leq i\leq n}\{r_i(B)\}$ and
$\max\limits_{1\leq i\leq n}\{r_i(B)\}=\phi_s$.
Note that  $\rho(B)=\max\limits_{1\leq i\leq n}\{r_i(B)\}$  if and only if the row sums of $B$, $r_1(B), \ldots, r_n(B)$ are all equal by Lemma \ref{lem21},
we have  $q(D)=\phi_s$ if and only if $\phi_1\geq \phi_s$ and $r_1(B)=\cdots=r_n(B)=\phi_s.$

Note that $r_1(B)=\cdots=r_n(B)=\phi_s$ if and only if $B$ satisfies the following four conditions:

(a) $x_i=1$ or $q_{ii}(=d_i^+)=d^+_1$ for $x_i>1$ holds for all $1\leq i\leq s-1$;

(b) $x_k=1$ or $q_{ik}=1$ for $x_k>1$ if $1\leq k\leq s-1$ with $k\neq i$ holds for all $1\leq i\leq s-1$,

(c)  $d_s^+=d_{s+1}^+=\cdots=d_n^+$; 

(d) $x_k=1$ or $q_{ik}=1$ for $x_k>1$ if $1\leq k\leq s-1$ holds for all $s\leq i\leq n$.

Thus we only need to show (a)--(d) hold if and only if $D$ is regular or
there exists an integer $t$ with $2\leq t\leq s$ such that $d^+_1=\cdots=d^+_{t-1}>d^+_t=\cdots=d^+_n$, and   $d^-_1=\cdots=d^-_{t-1}=n-1.$

If (a)-(d) hold, we consider the following cases.

{\bf Subcase 2.1: } $x_1=1$.

 Then   $x_1=x_2=\cdots=x_{s-1}=1$  by $x_1\geq x_2\geq \cdots\geq x_{s-1}\geq 1$, and thus $d_1^+=d_{2}^+=\cdots=d_{s-1}^+=d_s^+$. It implies that  $D$ is a regular digraph from (c).

{\bf Subcase 2.2: } $x_1\geq \cdots\geq x_{t-1}>1$ and $x_t=\cdots=x_{s-1}=1$ for some $t\in \{2,\ldots, s\}.$

Then $q_{ii}=d_i^+=d_1^+$ for $1\leq i\leq t-1$ by (a) and $d_{t}^+=\cdots=d_{s-1}^+=d_s^+=\cdots=d_n^+$ by (c).
Thus $d_1^+=\cdots=d_{t-1}^+> d_{t}^+=\cdots=d_n^+$.
By (b) and (d),  $q_{ik}=1$ (that is, $(i, k)\in E(D)$) for all $i\in\{1, 2, \ldots, n\}$ and all $k\in \{1, 2, \ldots, t-1\}\backslash\{i\}$,
That implies $d^-_1=\cdots=d^-_{t-1}=n-1$.

Conversely, if $D$ is a regular digraph, then $d_1^+=d_{2}^+=\cdots=d_n^+$ and $\phi_1=\cdots=\phi_n=2d_1^+$, the result follows.
If there exists some $t$ with $2\leq t\leq s$ such that $d_1^+=\cdots=d_{t-1}^+> d_{t}^+=\cdots=d_n^+$,
and $d^-_1=\cdots=d^-_{t-1}=n-1$, 
then $x_1\geq \cdots\geq x_{t-1}>1=x_t=\cdots=x_{s-1}$,
and   $(i,k)\in E(D)$ for all $i\in\{1, 2, \ldots, n\}$ and all $k\in \{1, 2, \ldots, t-1\}\backslash\{i\}$,
thus (a), 
 (b), 
 (c), 
 (d) hold.  
Therefore, $r_i(B)=\phi_s$ for all $i\in\{1, 2, \ldots, n\}$
and thus by Lemma \ref{lem21}, $q(D)=\rho(B)= \max\limits_{1\leq i\leq n}r_i(B)=\phi_s$.
\end{proof}

\begin{exam}\label{exam28}
For digraph $\overset{\longleftrightarrow}{K_n}-(u, v)$ where $u,v\in V(\overset{\longleftrightarrow}{K_n})$, by directly calculating, we see that
the bound of (\ref{eq21}) is better than the bounds of (\ref{eq11})--(\ref{eq14}) in \cite{2013ARS} (see Table 1) because

$$\frac{3n-6+\sqrt{n^2+4n-4}}{2}<\min\{2n-2, \frac{2n^2-4n+1}{n-1},  n-1+\sqrt{n(n-2)} \}.$$
\end{exam}

\begin{exam}\label{exam29}
Let $D_1$ as shown in Fig.1. For $D_1$, the outdegree sequence is $3=d^+_1>d^+_2=d^+_3=\cdots=d^+_n=2$ and the indegree $d^-_1=n-1$. Then $q(D_1)=3+\sqrt{3}$ by Theorem \ref{theo27}. We can see from Table 1 that the bound of (\ref{eq21}) is better than the bounds of (\ref{eq11})--(\ref{eq14}) in \cite{2013ARS} because
$q(D_1)=3+\sqrt{3}<\min\{5, \frac{5}{2}+\frac{\sqrt{21}}{2}, 3+\sqrt{6}\}.$
\end{exam}

However, we also see that the bound of (\ref{eq21}) is not  better than the bounds of (\ref{eq12})--(\ref{eq13}) in \cite{2013ARS} for any digraph.
  For example, let $D_2$ and $D_3$  as shown in Fig.1. For digraph $D_3$, the upper bounds of  (\ref{eq12})--(\ref{eq13}) and (\ref{eq21}) are equal to 3;
  and for digraph $D_2$,  the upper bounds of  (\ref{eq12})--(\ref{eq13}) are less than the upper bound of (\ref{eq21}).

 \vskip0.25cm
$
\xy 0;/r3pc/: \POS (4,1) *\xycircle<3pc,3pc>{};
         \POS (3.01, .9) *@{*}*+!R{u_n}="n";
         \POS (4.3, .9) *@{*}*+!U{u_1}="a";
         \POS(5.0,.9)  *@{*}*+!L{u_2}="b";
          \POS(4.8,1.6)  *@{*}*+!L{u_3}="c";
          \POS(4.4,1.9)  *@{*}*+!D{u_4}="d";
          \POS(4,2)  *@{*}*+!D{u_5}="e";
           \POS(3.5,1.86)   \ar@{->}(3.5,1.86);(3.6,1.91);
          \POS(3.3,1.72)  *@{*}*+!R{u_6}="f";
          \POS(3.4,1.3) *@{}*+!L{\vdots};
          \POS(2.8,1.4) *@{}*+!L{\vdots};
        \POS "b" \ar@{-} "n";  \POS "a" \ar@{-} "c"; \POS "a" \ar@{-} "d";
        \POS "e" \ar@{->} (4.15, 1.45); "e"; \POS (4.15, 1.45) \ar@{-} "a";
        \POS "f" \ar@{->} (3.8, 1.3); "f"; \POS (3.8, 1.3) \ar@{-} "a";
        \POS(3.8,.9) \ar@{->}(3.8,.9);(3.76,.9);
        \POS(4.99,1.1) \ar@{->}(4.98,1.25);(4.99,1.1); 
         \POS(3,1.1) \ar@{->}(3,1.1);(3.01,1.15); 
        \POS(4.5,1.85)   \ar@{->}(4.5,1.85);(4.6,1.8); 
        \POS(4.3,.06)   \ar@{->}(4.3,.06);(4.24,.038);
        \POS(4.23,1.98)   \ar@{->}(4.23,1.98);(4.27,1.965);
 \endxy
    \hskip0.5cm
 \xy 0;/r3pc/: \POS (3.5,1) *\xycircle<3pc,3pc>{};
        \POS(3.2,1.3) *@{}*+!D{\overset{\longleftrightarrow}{K_d}};
        \POS(3.5,0.3) *@{}*+!D{\overrightarrow{P_{n-d+2}}};
        \POS(3.7,-0.2) *@{}*+!D{\ldots};
        \POS(4.3,1.6)  *@{*}*+!L{\hspace*{3pt}{\hspace*{-3pt}\hspace*{-3pt}}}="c";
        \POS(5.4,1.6) *@{}*+!R{u_{n-d+2}};
        \POS(2.95,.16)   *@{*}*+!R{u_2}="d";
        \POS(2.7,.4)    \ar@{->}(2.7,.4) ;(2.69,.415) ="e";
        \POS(4.3,.4)   \ar@{->}(4.3,.4);(4.29,.385)="f";
         \POS (2.515, .9) *@{*}*+!R{u_1}="g";
         \POS "c" \ar @{-} "g";
         \POS(4.5,.9)  *@{*}*+!L{\hspace*{3pt}{\hspace*{-3pt}\hspace*{-3pt}}}="h";
         \POS(5.6,.9) *@{}*+!R{u_{n-d+1}};
        \POS(4.49,1.1) \ar@{->}(4.48,1.25);(4.49,1.1)="i";
\endxy
\hskip0.5cm
 \xy 0;/r3pc/: \POS (4,1) *\xycircle<3pc,3pc>{};
        \POS(4.8,1.6)  *@{*}*+!L{u_2}="c";
        \POS(4.75,.3)   *@{*}*+!L{u_n}="d";
         \POS (3.015, .9) *@{*}*+!R{u_{n-g+2}}="g";
         \POS(5.0,.9)  *@{*}*+!L{u_1}="h";
        \POS(4.99,1.1) \ar@{->}(4.98,1.25);(4.99,1.1)="i"; 
        \POS(3.5,1.86)   \ar@{->}(3.5,1.86);(3.6,1.91)="b"; 
        \POS(3.5,.145)    \ar@{->}(3.5,.145);(3.49,.151) ="e"; 
        \POS(3.2,0.15) *@{}*+!L{\ddots};
        \POS(4.94,.6)   \ar@{->}(4.94,.6);(4.93,.585)="f";
        \POS "h" \ar @{->} (3.8,.9) \ar @{-} "g";
        \POS (4.1, .1) *@{}*+!D{\overrightarrow{C_g}};
        \POS (4, 1.3) *@{}*+!D{\overrightarrow{P_{n-g+2}}};
        \POS (4, 1.9) *@{}*+!D{\cdots};
 \endxy
$

\hskip1.5cm $D_1$ \hskip3.8cm$D_2$  \hskip5.3cm$D_3$
\vskip0.1mm
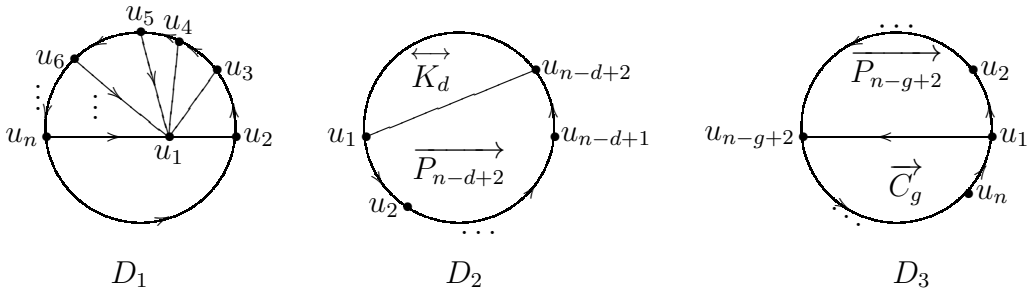
\hskip5cm  Fig.1. \hskip.1cm  The digraphs $D_1, D_2$ and $D_3$.

\vskip.2cm

\vskip.3cm
\hskip-0.8cm
\begin{tabular}{|c|c|c|c|c|c|}
 \hline
   \mbox{digraph}    & $(1.1) $    & $(1.2)$     & $(1.3)$     & $(1.4)$      & $(2.1)$  \\ \hline
$\overset{\longleftrightarrow}{K_n}-(u, v)$   & $2n-2$    & $\frac{2n^2-4n+1}{n-1}$     & $\frac{2n^2-4n+1}{n-1}$    & $n-1+\sqrt{n(n-2)}$    & $\frac{3n-6+\sqrt{n^2+4n-4}}{2}$  \\ \hline
$D_1$    & $5$    & $5$    & $\frac{5}{2}+\frac{\sqrt{21}}{2}$     & $3+\sqrt{6}$    & $3+\sqrt{3}$ \\ \hline
$D_2$    & $2d-1$    & $2d-2+\frac{2}{d}$    & $\frac{2d-1+\sqrt{1+\frac{4(d^2-2d+2)^2}{d(d-1)}}}{2}$     & $d+\sqrt{d^2-2d+2}$    & $\frac{3d-3+\sqrt{(d-1)^2+8}}{2}$ \\ \hline
$D_3$    & $3$    & $3$    & $3$     & $2+\sqrt{3}$    & $3$ \\ \hline
\end{tabular}
\vskip0.1mm
\hskip1.5cm  Table 1. \hskip.1cm  The upper bounds for digraphs $\overset{\longleftrightarrow}{K_n}-(u, v)$, $D_1, D_2$ and $D_3$.

\section{Some graph transformations on digraphs}

\hskip.6cm
In this section, we present some graph transformations on digraphs which are useful for the proof of the main results.

\begin{lem}\label{lem31}{\rm(\cite{1979}, Chapter 2)} 
Let $A$ be an $n\times n$ nonnegative matrix with the spectral radius $\rho(A)$,  $x=(x_1, x_2, \ldots, x_n)^T>0$ be any positive vector.
 If $\alpha\geq 0$  and $\alpha x\leq Ax$, then $\alpha\leq \rho(A)$.
  Furthermore, if $A$ is irreducible and $\alpha x< Ax$, then $\alpha< \rho(A)$.
 \end{lem}



In the rest of this section, let $x=(x_1, x_2, \ldots, x_n)^T$ be the unique positive  unit eigenvector corresponding to  $q(D)$,
while $x_i$ corresponds to the vertex $i$. 

\begin{theo}\label{theo32}
 Let $D=(V(D), E(D))$ be a simple digraph on $n$ vertices, $u, v, w\in V(D)$,  and $(u, v) \in E(D)$.
 Let  $H=D-\{(u, v) \}+\{(u, w) \}$ (Note that if $(u, w)\in E(D)$, then $H$ has multiple arc $(u, w)$.)
 If $x_w\geq x_v$, then $q(H)\geq q(D)$.   Furthermore, if $H$ is strongly connected and $x_w>x_v$, then $q(H)>q(D)$.
\end{theo}
\begin{proof}
Now we show $(Q(H)x)_s\geq (Q(D)x)_s$  for any $s\in V(D)=V(H)$.




When $s\neq u$, then  $(Q(H)x)_s=\sum\limits_{t=1}^{n}q_{st}x_t=(Q(D)x)_s=q(D)x_s$ where $Q(D)=(q_{ij})$; 
when $s= u$, then  $(Q(H)x)_s-(Q(D)x)_s=x_w-x_v\geq 0.$
Thus $Q(H)x\geq Q(D)x=q(D)x$. By Lemma \ref{lem31}, $\rho(Q(H))=q(H)\geq q(D)$.

Similarly, if $H$ is strongly connected and $x_w>x_v$, then $q(H)>q(D)$  by Lemma \ref{lem31} immediately.
\end{proof}

\begin{lem}\label{lem33}{\rm(\cite{1979}, Chapter 2; \cite{1988}, Chapter 1)} 
Let $A$ be an $n\times n$ nonnegative matrix with the spectral radius $\rho(A)$.
Then $A$ is reducible if and only if $\rho(A)$ is the spectral radius of some proper principal submatrix of $A$.
\end{lem}

The following result follows from Lemma \ref{lem33} in terms of digraph.

\begin{cor}\label{cor34}
Let $D$ be a digraph and $D_1, D_2, \ldots, D_s$ be the strongly connected components of
$D$. Then $q(D)=\max\{q(D_1), q(D_2), \ldots, q(D_s)\}$.
\end{cor}

\begin{lem}\label{lem35}
Let $D$ $(\neq \overrightarrow{C_n})$ be a strongly connected digraph with $V(D)=\{u_1, u_2, \ldots, u_n\}$,
 $\overrightarrow{P}=u_1u_2\cdots u_l$ $(l\geq3)$ be a directed path of $D$ with $d^+_{u_i}=1$ $(i=2, 3, \ldots, l-1)$.
  Then we have $x_2<x_3<\cdots<x_{l-1}<x_l$.
\end{lem}

\begin{proof}
Since $D$ is a strongly connected digraph and $D\neq \overrightarrow{C_n}$, then $D$ contains a directed cycle, denoted by
$\overrightarrow{C_g}$ $(g\geq 2)$, as a proper subdigraph of $D$.
Thus $q(D)>q(\overrightarrow{C_g})=2$ by Corollary \ref{cor25}.
Therefore, for any $i\in \{2,3,\ldots, l-1\}$,  we have

\hskip2cm $2x_i<q(D)x_i=(Q(D)x)_i=d^+_{u_i}x_i+x_{i+1}=x_i+x_{i+1}.$

\noindent  Then $x_i<x_{i+1}$ and thus $x_2<x_3<\cdots<x_{l-1}<x_l$.
\end{proof}

Let $D_{uv}$ denote the simple digraph obtained from $D$ by deleting  arc $(u, v)$, identifying $u$ with $v$ of $D$ and deleting the multiple arcs.

\begin{theo}\label{theo36}
Let $D$ $(\neq \overrightarrow{C_n})$ be a strongly connected digraph with $V(D)=\{u_1, u_2, \ldots, u_n\}$,
and  $\overrightarrow{P}=u_1u_2\cdots u_l$ $(l\geq3)$ be a directed path of $D$ with $d^+_{u_i}=d^-_{u_i}=1$ $(i=2, 3, \ldots, l-1)$.
 Then for any $i\in \{2, 3, \ldots, l-1\}$,  $q(D_{u_{i-1}u_i})\geq q(D)$.
\end{theo}

\begin{proof}
For any $i\in \{2, 3, \ldots, l-1\}$, let $H=D-\{(u_{i-1}, u_i)\}+\{(u_{i-1}, u_{i+1})\}$.
Then by $d^-_{u_i}=1$, $H$ has exactly two strongly connected components, one is the isolated vertex $u_i$, the other is $D_{u_{i-1}u_i}$,
thus $q(H)=q(D_{u_{i-1}u_i})$ by Corollary \ref{cor34}.

On the other hand, for any $i\in \{2, 3, \ldots, l-1\}$, by Lemma \ref{lem35} and $d^+_{u_i}=1$, we have  $x_{i+1}>x_i$. Then  $q(H) \geq q(D)$ by Theorem \ref{theo32},
and thus  $q(D_{u_{i-1}u_i})=q(H)\geq q(D)$.
\end{proof}

Let $D=(V(D), E(D))$ be a digraph with $(u, v)\in E(D)$ and $w\notin V(D)$,
 $D^w=(V(D^w), E(D^w))$ with $V(D^w)=V(D)\cup \{w\}$, $E(D^w)=E(D)-\{(u, v)\}+\{(u, w), (w, v)\}$.
 Then the following result follows from Theorem \ref{theo36}.

\begin{cor}\label{cor37}
Let $D$ $(\neq \overrightarrow{C_n})$ be a strongly connected digraph, $w\notin V(D)$, and $D^w$ defined as before. Then $q(D)\geq q(D^w)$.
\end{cor}
\begin{proof}
Clearly $D=D^w_{uw}$, $D^w(\neq \overrightarrow{C_n})$ is a strongly connected digraph,  $\overrightarrow{P}=uwv$ is a directed path of $D^w$ and
the outdegree and the indegree of $w$ in $D^w$ are equal to 1, then  $q(D)=q(D^w_{uw})\geq q(D^w)$ by Theorem \ref{theo36}.
\end{proof}

\section{The minimum signless Laplacian spectral radius of digraphs with given clique number}
\hskip.6cm
Let $\mathcal{C}_{n,d}$ denote the set of strongly connected digraphs on $n$ vertices with clique number $\omega(D)=d$. As we know, if $d=n$, then $\mathcal{C}_{n,n}=\{\overset{\longleftrightarrow}{K_n}\}$ and $q(\overset{\longleftrightarrow}{K_n})=2n-2$. If $d=1$, then $\overrightarrow{C_n}\in\mathcal{C}_{n,1}$ and $q(\overrightarrow{C_n})=2$. By Corollary \ref{cor26}, for any $D\in \mathcal{C}_{n,1}$, $q(D)\geq q(\overrightarrow{C_n})=2$ with equality if and only if $D\cong \overrightarrow{C_n}$. Thus we only  discuss the cases $2\leq d\leq n-1$.

Let $2\leq d\leq n-1$, $B_{n,d}=(V(B_{n,d}), E(B_{n,d}))$ be a digraph obtained by adding a directed
\vskip.1cm
\noindent path $\overrightarrow{P_{n-d+2}}=u_1u_2\cdots u_{n-d+2}$ to a clique $\overset{\longleftrightarrow}{K_d}$ such that $V(\overset{\longleftrightarrow}{K_d})\cap V(\overrightarrow{P_{n-d+2}})=\{u_1, u_{n-d+2}\}$ (as shown in Fig.2), where $V(B_{n,d})=\{u_1, u_2, \ldots, u_n\}$. Clearly, $B_{n,d}\in \mathcal{C}_{n,d}$. In this section, we will show that $B_{n, d}$ is the unique digraph with the minimum signless Laplacian spectral radius among all digraphs in $\mathcal{C}_{n,d}$ where $2\leq d\leq n-1$.

\vskip.2cm
$
        \hskip3cm
 \xy 0;/r3pc/: \POS (3.5,1) *\xycircle<3pc,3pc>{};
        \POS(3.2,1.3) *@{}*+!D{\overset{\longleftrightarrow}{K_d}};
        \POS(3.5,0.4) *@{}*+!D{\overrightarrow{P_{n-d+2}}};
        \POS(4.3,1.6)  *@{*}*+!L{\hspace*{3pt}{\hspace*{-3pt}\hspace*{-3pt}}}="c";
        \POS(5.4,1.6) *@{}*+!R{u_{n-d+2}};
        \POS(2.95,.16)   *@{*}*+!R{u_2}="d";
               \POS(2.7,.4)    \ar@{->}(2.7,.4) ;(2.69,.415) ="e";
        \POS(4.3,.4)   \ar@{->}(4.3,.4);(4.29,.385)="f";
         \POS (2.515, .9) *@{*}*+!R{u_1}="g";
         \POS "c" \ar @{-} "g";
         \POS(4.5,.9)  *@{*}*+!L{\hspace*{3pt}{\hspace*{-3pt}\hspace*{-3pt}}}="h";
         \POS(5.6,.9) *@{}*+!R{u_{n-d+1}};
        \POS(4.49,1.1) \ar@{->}(4.48,1.25);(4.49,1.1)="i";
        \POS(3.7,-0.2) *@{}*+!D{\ldots};
 \endxy
        \hskip1cm
\xy 0;/r3pc/: \POS (1,1) *\xycircle<2pc,2pc>{}; \POS (2.33,1)*\xycircle<2pc,2pc>{}
        \POS(1.665,1) *@{*}*+!R{u_1}="n";
        \POS(1.28,0.38) *@{*}*+!U{u_{n-d+2}}="k";
        \POS(0.8,0.8) *@{}*+!D{\overset{\longleftrightarrow}{K_d}};
        \POS(2.1,1.6)   *@{*}*+!D{u_{n-d+1}}="d";
        \POS(1.8,.58)    \ar@{->}(1.8,.58) ;(1.79,.6) ="e";
        \POS (2.33, .8) *@{}*+!D{\overrightarrow{C_{n-d+1}}};
        \POS (3.1, .8) *@{}*+!D{\vdots};
         \POS(2.0,.4)  *@{*}*+!U{u_2}="h";
        \POS(2.98,0.9)   \ar@{->}(2.98,1.15);(2.99,1)="j";
        \POS(1.78,1.37) \ar@{->}(1.78,1.37);(1.79,1.385)="i";
\endxy
$
\vskip0.00001mm
\hskip4.5cm $B_{n,d}$ \hskip4.3cm  $B'_{n,d}$
\vskip0.1mm
\hskip4.5cm  Fig.2. \hskip.1cm  The digraphs $B_{n,d}$ and $B'_{n,d}$.

\begin{lem}\label{lem41}
Let $B'_{n,d}=B_{n,d}-\{(u_{n-d+1},u_{n-d+2})\}+\{(u_{n-d+1},u_1)\}$ (as shown in Fig.2). Then $q(B'_{n,d})>q(B_{n,d})$.
\end{lem}

\begin{proof}
Clearly, $B'_{n,d}$ is strongly connected. Let  $x=(x_1, x_2, \ldots, x_n)^T$ be the unique positive unit eigenvector corresponding to $q=q(B_{n,d})$, while $x_i$ corresponds to the vertex $u_i$. By Theorem \ref{theo32}, we only need to show $x_1>x_{n-d+2}$.

Since $\overset{\longleftrightarrow}{K_d}$ is a  proper subdigraph of $B_{n,d}$,
then  $q=q(B_{n,d})>q(\overset{\longleftrightarrow}{K_d})=2d-2$ by Corollary \ref{cor25}.
Let $V_1=V(\overset{\longleftrightarrow}{K_d})\backslash\{u_1, u_{n-d+2}\}$. We have

$(Q(B_{n,d})x)_1=qx_1=dx_1+\sum\limits_{v\in V_1}x_v+x_2+x_{n-d+2}, $

\noindent  and

 $(Q(B_{n,d})x)_{n-d+2}=qx_{n-d+2}=(d-1)x_{n-d+2}+\sum\limits_{v\in V_1}x_v+x_1$.

Then $(q-d+1)(x_1-x_{n-d+2})=x_2+x_{n-d+2}>0$. Thus $x_1>x_{n-d+2}$ by $q>2d-2$.
\end{proof}

\begin{theo}\label{theo42}
Let $2\leq d\leq n-1$ and $D\in\mathcal{C}_{n,d}$ be a digraph. Then $q(D)\geq q(B_{n,d})$, with equality if and only if $D\cong B_{n,d}$.
\end{theo}

\begin{proof}
 Clearly,  $\overset{\longleftrightarrow}{K_d}$ is a proper subdigraph of $D$ because of $D\in \mathcal{C}_{n,d}$.
 Since $D$ is strongly connected, then delete the vertices or arcs of $D$ such that the resulting digraph is denoted by $H$,
 while $H\cong B_{d+l-2, d}$ $(l\geq 3)$ or $H\cong B'_{d+l-1, d}$ $(l\geq 2)$.
By Corollary \ref{cor25}, $q(H)\leq q(D)$ with equality if and only if $H\cong D$.

{\bf Case 1: } $H\cong B_{d+l-2, d}$ $(l\geq 3)$.

 Insert $n-d-l+2$ vertices into $\overrightarrow{P_l}$ such that the resulting digraph is $B_{n, d}$.
 Then  $q(B_{n,d})\leq q(H)$ by using Corollary \ref{cor37} $n-d-l+2$ times.

{\bf Case 2: }  $H\cong B'_{d+l-1, d}$ $(l\geq 2)$.

Insert $n-d-l+1$ vertices into the directed cycle $\overrightarrow{C_l}$ such that the resulting digraph is $B'_{n, d}$.
Then  $q(B'_{n,d})\leq q(H)$ by using Corollary \ref{cor37}  $n-d-l+1$ times, and thus $q(B_{n,d})<q(B'_{n,d})\leq q(H)$ by Lemma \ref{lem41}.

Combining the above two cases, we have $q(D)\geq q(B_{n,d})$ with equality if and only if $D\cong B_{n,d}$.
\end{proof}

Now we  estimate the signless Laplacian spectral radius of $B_{n,d}$ and show $2=q(\overrightarrow{C_n})<q(B_{n,2})<q(B_{n,3})<\cdots<q(B_{n,n-1})<q(\overset{\longleftrightarrow}{K_n})=2n-2$.

\begin{lem}\label{lem43}
Let $2\leq d\leq n-1$ and $B_{n,d}$ be defined as above. Then
$2d-2<q(B_{n,d})\leq\frac{3d-3+\sqrt{(d-1)^2+8}}{2}$.
\end{lem}

\begin{proof}
Clearly,  $q(B_{n,d})>q(\overset{\longleftrightarrow}{K_d})=2d-2$ since Corollary \ref{cor25} and $\overset{\longleftrightarrow}{K_d}$ is a proper subdigraph of $B_{n,d}$.

Let $x=(x_1, x_2, \ldots, x_n)^T$ be the unique positive unit eigenvector  corresponding to $q(B_{n,d})$,
while $x_i$ corresponds to he vertex $u_i$. Then  $x_{2}<x_{3}<\cdots<x_{n-d+2}$ by Lemma \ref{lem35}.

Let $H=B_{n, d}-\{(u_1, u_2)\}+\{(u_1, u_{n-d+2})\}$. Then  $q(H)\geq q(B_{n, d})$ by Theorem \ref{theo32}.
It is easy to check that $H$ has $n-d+1$ strongly connected components, one is $H_1=\overset{\longleftrightarrow}{K_d}\cup\{(u_1, u_{n-d+2)}\}$ which has multiple arcs $(u_1, u_{n-d+2})$,  and the others are isolated vertices $u_2, u_3, \ldots, u_{n-d+1}$.
Then $q(H)=q(H_1)$ by Corollary \ref{cor34}.

Let $q=q(H_1)$ and $y=(y_1, y_2, \ldots, y_d)^T$ be the unique positive unit eigenvector corresponding to $q$.
Then for any two vertices $u, v\in V(\overset{\longleftrightarrow}{K_d})\backslash\{u_1\}$,
let $V_1=V(\overset{\longleftrightarrow}{K_d})\backslash\{u, v\}$, we have

$(Q(H_1)y)_u=qy_u=(d-1)y_u+\sum\limits_{w\in V_1}y_w+y_v$,
$(Q(H_1)y)_v=qy_v=(d-1)y_v+\sum\limits_{w\in V_1}y_w+y_u$.

Thus $y_u=y_v$, it implies $y_s=y_u$ for any $s\in V(\overset{\longleftrightarrow}{K_d})\backslash\{u_1\}$ by the choice of $u$ and $v$. Then

\vskip.2cm
$\left\{
\begin{array}{c}
qy_u=(2d-3)y_u+y_{u_1}, \\[0.2cm]
qy_{u_1}=dy_{u_1}+dy_u.
\end{array}
\right.$
\vskip.2cm

Then we have $q^2-(3d-3)q+(2d^2-4d)=0$, and $q=\frac{3d-3+\sqrt{(d-1)^2+8}}{2}$ by $q=q(H_1)=q(H)\geq q(B_{n, d})>2d-2$.
Thus $q(B_{n,d})\leq\frac{3d-3+\sqrt{(d-1)^2+8}}{2}$.
\end{proof}

\begin{rem}
Since the outdegree sequence of $B_{n,d}$ is $d^+_1=d, d^+_2=d^+_3=\cdots=d^+_{d}=d-1$, $d^+_{d+1}=d^+_{d+2}=\cdots=d^+_{n}=1$, we can also proof that $q(B_{n,d})\leq\frac{3d-3+\sqrt{(d-1)^2+8}}{2}=\phi_2=\ldots=\phi_d$ by Theorem \ref{theo27}.
\end{rem}

From Lemma \ref{lem43}, we immediately get the following corollary.
\begin{cor}\label{cor44}Let $n\geq 4$. Then 
$2=q(\overrightarrow{C_n})<q(B_{n,2})<4<q(B_{n,3})<6<\cdots<2n-4<q(B_{n,n-1})<q(\overset{\longleftrightarrow}{K_n})=2n-2$.
\end{cor}
\begin{proof}
Since $2d-2<\frac{3d-3+\sqrt{(d-1)^2+8}}{2}\leq 2d$, then $2d-2<q(B_{n,d})<2d<q(B_{n,d+1})<2d+2$ for $2\leq d\leq n-1$ by Lemma \ref{lem43}.
\end{proof}


\section{The minimum signless Laplacian spectral radius of digraphs with given girth}

\hskip.6cm
Let $g\geq 2$ and $\mathcal{G}_{n, g}$ denote the set of strongly connected digraph on $n$ vertices with girth $g$.
If $g=n$, then $\mathcal{G}_{n, n}=\{\overrightarrow{C_n}\}$ and $q(\overrightarrow{C_n})=2$.
Thus we only need to discuss the cases $2\leq g\leq n-1$.

 Let $2\leq g\leq n-1$ and $C_{n, g}=(V(C_{n,g}), E(C_{n,g}))$ be a digraph obtained by adding a directed path $\overrightarrow{P_{n-g+2}}=u_gu_{g+1}\cdots u_{n}u_1$ on the directed cycle $\overrightarrow{C_g}=u_1u_2\cdots u_gu_1$ such that $V(\overrightarrow{C_g})\cap V(\overrightarrow{P_{n-g+2}})=\{u_1, u_{g}\}$ (as shown in Fig.3), where $V(C_{n,g})=\{u_1, u_2, \cdots, u_n\}$, and $E(C_{n, g})=\{(u_i,u_{i+1}), 1\leq i\leq n-1\}\cup\{(u_g, u_1), (u_n, u_1)\}$.
 Clearly, $C_{n, g}\in \mathcal{G}_{n, g}$. In the rest of this section, we will show that $C_{n, g}$ achieves the minimum signless Laplacian spectral radius among all digraphs in $\mathcal{G}_{n, g}$. We also determine the digraphs which achieve the second, the third and the fourth minimum signless Laplacian spectral radius among all digraphs.
 \vskip0.25cm
$
\hskip3cm
 \xy 0;/r3pc/: \POS (4,1) *\xycircle<3pc,3pc>{};
        \POS(3.2,1.6)  *@{*}*+!L{u_{n}}="j";
          \POS(4.8,1.6)  *@{*}*+!L{u_{g+1}}="c";
        \POS(4.75,.32)   *@{*}*+!L{u_{g-1}}="d";
         \POS (3.015, .9) *@{*}*+!R{u_{1}}="g";
         \POS(5.0,.9)  *@{*}*+!L{u_g}="h";
        \POS(4.99,1.1) \ar@{->}(4.98,1.25);(4.99,1.1)="i";
       \POS(3,1.1) \ar@{->}(3,1.1);(3.01,1.15); 
        \POS(3.5,1.86)   \ar@{->}(3.5,1.86);(3.6,1.91)="b";
        \POS(3.5,.145)    \ar@{->}(3.5,.145) ;(3.49,.151) ="e";
        \POS(4.94,.6)   \ar@{->}(4.94,.6);(4.93,.585)="f";
        \POS "h" \ar @{->} (3.8,.9) \ar @{-} "g";
        \POS (4.1, .1) *@{}*+!D{\overrightarrow{C_g}};
         \POS (4, -0.25) *@{}*+!D{\cdots};
        \POS (4.1, 1.1) *@{}*+!D{\overrightarrow{P_{n-g+2}}};
        \POS (4, 1.9) *@{}*+!D{\cdots};
 \endxy
        \hskip0.2cm
    \hskip1cm
\xy 0;/r3pc/: \POS (1,1) *\xycircle<2pc,2pc>{}; \POS (2.33,1)*\xycircle<2pc,2pc>{}
        \POS(1.665,1) *@{*}*+!R{u_g}="n";
        \POS(1.28,0.38) *@{*}*+!U{u_{1}}="k";
        \POS(0.8,0.8) *@{}*+!D{\overrightarrow{C_{g}}};
        \POS(0.22,0.8) *@{}*+!D{\vdots};
        \POS(1.6,1.3) \ar@{->}(1.6,1.3);(1.55,1.37)="a";
        \POS(1.45,0.5) \ar@{->}(1.45,0.5);(1.5,0.55);
        \POS(0.33,0.9)   \ar@{->}(0.33,1);(0.34,0.85);
        \POS(2.1,1.6)   *@{*}*+!D{u_{n}}="d";
        \POS(1.8,.58)    \ar@{->}(1.8,.58) ;(1.79,.6) ="e";
        \POS (2.33, .8) *@{}*+!D{\overrightarrow{C_{n-g+1}}};
         \POS(2.0,.4)  *@{*}*+!U{u_{g+1}}="h";
        \POS(2.98,0.9)   \ar@{->}(2.98,1.15);(2.99,1)="j";
        \POS(1.78,1.37) \ar@{->}(1.78,1.37);(1.79,1.385)="i";
        \POS (3.1, .8) *@{}*+!D{\vdots};
\endxy
$

\vskip0.00001mm
\hskip4.8cm $C_{n,g}$ \hskip4cm  $C'_{n,g}$
\vskip0.1mm
\hskip5cm  Fig.3. \hskip.1cm  The digraphs $C_{n,g}$ and $C'_{n,g}$.

\begin{lem}\label{lem51}
Let $2\leq g\leq n-1$, $C'_{n,g}=C_{n, g}-\{(u_{n},u_1)\}+\{(u_{n}, u_g)\}$ (as shown in Fig.3) and
$x=(x_1, x_2, \ldots, x_n)^T$ be the unique positive unit eigenvector corresponding to $q(C_{n,g})$,
while $x_i$ corresponds to the vetex $u_i$. Then

{\rm (1) } $x_{g+1}<x_{g+2}<\cdots <x_n<x_1<x_2<\cdots<x_g$;

{\rm (2) } $q(C'_{n, g})>q(C_{n, g})$.
\end{lem}

\begin{proof}
Since $\overrightarrow{Q}=u_{g+1}u_{g+2}\cdots u_{n}u_1\cdots u_g$ and $\overrightarrow{R}=u_{g}u_{g+1}u_{g+2}$ are  directed paths of $C_{n,g}$
with $d^+_{u_i}=1$  where $i\in \{1,2, \ldots, g-1, g+1, \ldots, n\}$, then by Lemma \ref{lem35}, we have
$x_{g+2}<\cdots <x_{n}<x_1<\cdots<x_g$ and $x_{g+1}<x_{g+2}$, thus $x_{g+1}<x_{g+2}<\cdots <x_n<x_1<x_2<\cdots<x_g$.

Since $C'_{n,g}$ is strongly connected and $x_1<x_g$, then $q(C'_{n, g})>q(C_{n, g})$ by Theorem \ref{theo32}.
\end{proof}

\begin{theo}\label{theo52}
Let $2\leq g\leq n-1$ and $D\in\mathcal{G}_{n, g}$ be a digraph. Then $q(D)\geq q(C_{n, g})>2$, with equality if and only if $D\cong C_{n, g}$.
\end{theo}
\begin{proof}
Since $D\in \mathcal{G}_{n, g}$, then $\overrightarrow{C_g}$ is the proper subdigraph of $D$ and thus $q(D)>2=q(C_g)$. Without loss of generality,
we let $\overrightarrow{C_g}=u_1u_2\cdots u_gu_1$. Since $D\in\mathcal{G}_{n, g}$ is strongly connected, then
delete vertices or arcs of $D$ such that the resulting subdigraph is denoted by $D_1$, while $D_1\cong C'_{g+l-1, g}$ $(l\geq g)$
or $D_1\cong H$, where $H=(V(H), E(H))$,  $V(H)=\{u_1, \ldots, u_g, u_{g+1}, \ldots, u_{g+l-2}\}$,
$E(H)=\{(u_i, u_{i+1})| i\in \{1, \ldots,  g+l-3\}\}\cup\{(u_g, u_1), (u_{g+l-2}, u_t)\}$ with $1\leq t\leq g$, $1+t\leq l$ (see Fig.4).

By Corollary \ref{cor25}, we have $q(D)\geq q(D_1)$ with equality if and only if $D\cong D_1$.

 \vskip0.25cm
$
\hskip3cm
 \xy 0;/r3pc/: \POS (4,1) *\xycircle<3pc,3pc>{};
        \POS(3.2,1.6)  *@{*}*+!R{u_{g+l-2}}="c";
         \POS (3.01, .9) *@{*}*+!R{u_{t}}="g";
         \POS (4.3, .9) *@{*}*+!D{u_{1}};
         \POS(5.0,.9)  *@{*}*+!L{u_g}="h";
         \POS(4.9,1.4)  *@{*}*+!L{u_{g+1}};
             \POS(4.5,.9) \ar@{->}(4.5,.9);(4.504,.9); 
             \POS(3,1.1) \ar@{->}(3,1.1);(3.01,1.15)="i";
        \POS(4.5,1.85)   \ar@{->}(4.5,1.85);(4.6,1.8)="b";
        \POS(4.94,.6)   \ar@{->}(4.94,.6);(4.93,.585)="f";
        \POS(4.3,.06)   \ar@{->}(4.3,.06);(4.24,.038);
        \POS "h" \ar@{-} "g";
        \POS(3.8,.9) \ar@{->}(3.8,.9);(3.76,.9)
        \POS (4.1, .08) *@{}*+!D{\overrightarrow{C_g}};
        \POS (4, 1.3) *@{}*+!D{\overrightarrow{P_{l}}};
        \POS(4.99,1.1) \ar@{->}(4.98,1.25);(4.99,1.1); 
        \POS (4, 1.9) *@{}*+!D{\cdots};
        \POS(2.9,1) *@{}*+!D{\vdots};
        \POS (4, -0.25) *@{}*+!D{\cdots};
        \POS (3.6, .85) *@{}*+!D{\cdots};

 \endxy
        \hskip0.2cm
    \hskip1cm
\xy 0;/r3pc/: \POS (1,1) *\xycircle<2pc,2pc>{}; \POS (2.33,1)*\xycircle<2pc,2pc>{}
        \POS(1.665,1) *@{*}*+!R{u_g}="n";
        \POS(1.28,0.38) *@{*}*+!U{u_{1}}="k";
        \POS(0.8,0.8) *@{}*+!D{\overrightarrow{C_{g}}};
        \POS(1.6,1.3) \ar@{->}(1.6,1.3);(1.55,1.37)="a";
        \POS(1.45,0.5) \ar@{->}(1.45,0.5);(1.5,0.55);
        \POS(0.33,0.9)   \ar@{->}(0.33,1);(0.34,0.85);
        \POS(2.1,1.6)   *@{*}*+!D{u_{g+l-1}}="d";
        \POS(1.8,.58)    \ar@{->}(1.8,.58) ;(1.79,.6) ="e";
        \POS (2.33, .8) *@{}*+!D{\overrightarrow{C_{l}}};
         \POS(2.0,.4)  *@{*}*+!U{u_{g+1}}="h";
        \POS(2.98,0.9)   \ar@{->}(2.98,1.15);(2.99,1)="j";
        \POS(1.78,1.37) \ar@{->}(1.78,1.37);(1.79,1.385)="i";
        \POS(0.22,0.8) *@{}*+!D{\vdots};
        \POS (3.1, .8) *@{}*+!D{\vdots};
\endxy
$

\vskip0.00001mm
\hskip4.8cm $H$ \hskip4cm  $C'_{g+l-1,g}$
\vskip0.1mm
\hskip5cm  Fig.4. \hskip.1cm  The digraphs $H$ and $C'_{g+l-1,g}$.

{\bf Case 1:  } $D_1\cong C'_{g+l-1, g}$ where  $l\geq g$.

Insert $n-g-l+1$ vertices into $\overrightarrow{C_l}$ such that the resulting digraph is $C'_{n, g}$.
By using Corollary \ref{cor37} $n-g-l+1$ times, we have $q(D_1)\geq q(C'_{n, g})$,
 and thus $q(D)\geq q(D_1)\geq q(C'_{n, g})>q(C_{n, g})$ by Lemma \ref{lem51}.

{\bf Case 2: }  $D_1\cong H$.

 Insert $n-g-l+2$ vertices to $\overrightarrow{P_l}$ such that the resulting digraph, denoted by $H'$.
 Clearly, $H^{\prime}$ is strongly connected. By using Corollary \ref{cor37} $n-g-l+2$ times,
 we have $q(H)\geq q(H^{\prime})$ with equality if and only if $H\cong H^{\prime}$.

{\bf Subcase 2.1: } $t=1$.

In this case, $H^{\prime}\cong C_{n,g}$, then $q(D)\geq q(D_1)=q(H)\geq q(H^{\prime})=q(C_{n,g})$,
with equality if and only if $D\cong C_{n, g}$.

 {\bf Subcase 2.2: } $2\leq t\leq g$.

Note that $H^{\prime}\cong C_{n,g}-\{(u_n, u_1)\}+\{(u_n, u_t)\}$,
By Lemma \ref{lem51}, we have $x_1<x_t$,  and then $q(C_{n, g})<q(H^{\prime})$ by Theorem \ref{theo32}.
Thus  $q(D)\geq q(D_1)=q(H)\geq q(H^{\prime})>q(C_{n,g})$.

Combining the above arguments,  $q(D)\geq q(C_{n, g})$ with equality if and only if $D\cong C_{n, g}$.
\end{proof}

\begin{theo}\label{theo53} Let $n\geq 4$. Then 
$2=q(\overrightarrow{C_n})<q(C_{n,n-1})<q(C_{n,n-2})<\cdots<q(C_{n,2})<3$.
\end{theo}
\begin{proof}
It is clear $ q(\overrightarrow{C_n})=2$ and $q(C_{n,n-1})>2$.

For any $g$ with $2\leq g\leq n-1$, note that the outdegree of $C_{n, g}$ is $d^+_1=2, d^+_2=d^+_3=\cdots=d^+_{n}=1$,
then $q(C_{n,g})<3$ by Theorem \ref{theo27}.

Now  we only need to show that $q(C_{n,g+1})<q(C_{n,g})$ for $2\leq g\leq n-2$.

Let $x=(x_1, x_2, \ldots, x_n)^T$ be the unique positive unit eigenvector corresponding to $q(C_{n,g+1})$,
while $x_i$ corresponds to the vertex $u_i$. Clearly, $C_{n,g}\cong C_{n,g+1}-\{(u_{g+1}, u_{1})\}+\{(u_{g+1}, u_{2})\}$.
Similar to the proof of lemma \ref{lem51}, we have $x_1<x_2$, then $q(C_{n,g+1})<q(C_{n,g})$  by Theorem \ref{theo32}.
\end{proof}

In \cite{2012LAA}, the authors defined  $\theta$-digraph as follows. The $\theta$-digraph consists of three directed paths $P_{a+2}$, $P_{b+2}$ and $P_{c+2}$ such that the initial vertex of $P_{a+2}$ and $P_{b+2}$ is the terminal vertex of $P_{c+2}$, and the initial vertex of $P_{c+2}$ is the terminal vertex of $P_{a+2}$ and $P_{b+2}$, denoted by $\theta(a,b,c)$. Clearly, $C_{n,g}\cong \theta(0,n-g,g-2)$, and $H\cong\theta(t-1, l-2, g-t-1)$ where $H$ defined in the proof of Theorem \ref{theo52}.

Let $\theta(1,1,n-4), \widehat{\theta}=\theta(1,1,n-4)\cup\{(u_2, u_3)\}, \theta(0,2,n-4)$ as shown in Fig.5. By Corollary \ref{cor26},
we know that $\overrightarrow{C_n}$ is the digraph which achieve the minimum signless Laplacian spectral radius
among all strongly connected digraphs on $n\geq 4$ vertices.
Now we will show $\theta(0,1,n-3)$ (that is, $C_{n,n-1}$), $\theta(1,1,n-4) $, $\theta(0,2,n-4)$ (that is, $C_{n,n-2}$) are the digraphs which achieve the second,
the third and the fourth minimum signless Laplacian spectral radius among all strongly connected digraphs on $n\geq 4$ vertices, respectively.

 \vskip0.2cm
$
\hskip1cm
 \xy 0;/r3pc/: \POS (4,1) *\xycircle<3pc,3pc>{};
        \POS(4.1,2)  *@{*}*+!D{u_2}="c";
         \POS (3.01, .9) *@{*}*+!R{u_{4}}="g";
         \POS (4.3, .9) *@{*}*+!D{u_3};
         \POS(5.0,.9)  *@{*}*+!L{u_1}="h";
          \POS(4.5,.9) \ar@{->}(4.5,.9);(4.504,.9);
        \POS(3.25, 1.65) \ar@{->}(3.25, 1.65);(3.2, 1.58)
        \POS(4.8,1.57)   \ar@{->}(4.8,1.57);(4.9,1.44);
        \POS(4.94,.6)   \ar@{->}(4.94,.6);(4.93,.585);
                \POS(4.3,.06)   \ar@{->}(4.3,.06);(4.24,.038);
        \POS "h" \ar@{-} "g";
        \POS(3.8,.9) \ar@{->}(3.8,.9);(3.76,.9)
        \POS (4, -0.25) *@{}*+!D{\cdots};
        \POS(3.2,.4)    \ar@{->}(3.2,.4) ;(3.19,.415) ;
        \POS (3.4, .2) *@{*}*+!R{u_5};
        \POS (4.6, .2) *@{*}*+!L{u_n};
 \endxy
 \hskip1.2cm
 \xy 0;/r3pc/: \POS (4,1) *\xycircle<3pc,3pc>{};
        \POS(4.1,2)  *@{*}*+!D{u_3}="c";
         \POS (3.15, 1.5) *@{*}*+!R{u_{4}}="d";
         \POS (4.85, 1.5) *@{*}*+!L{u_2}="b";
         \POS(5.0,.9)  *@{*}*+!L{u_1}="a";
       \POS(3.5,1.86)   \ar@{->}(3.5,1.86);(3.6,1.91);
        \POS(4.99,1.1) \ar@{->}(4.98,1.25);(4.99,1.1);
         \POS(4.5,1.85)   \ar@{->}(4.5,1.85);(4.6,1.8);
        \POS(4.94,.6)   \ar@{->}(4.94,.6);(4.93,.585);
        \POS(4.3,.06)   \ar@{->}(4.3,.06);(4.24,.038);
        \POS "a" \ar@{-} "c";  \POS "b" \ar@{-} "d";
           \POS(3.8,1.5) \ar@{->}(3.8,1.5);(3.76,1.5)
           \POS(3,1.1) \ar@{->}(3,1.1);(3,1)
           \POS(4.75,1.2)   \ar@{->}(4.75,1.2);(4.7,1.25)
        \POS (4, -0.25) *@{}*+!D{\cdots};
        \POS(3.2,.4)    \ar@{->}(3.2,.4) ;(3.19,.415) ;
        \POS (3, .9) *@{*}*+!R{u_5};
        \POS (4.6, .2) *@{*}*+!L{u_n};
 \endxy
 \hskip1.2cm
 \xy 0;/r3pc/: \POS (4,1) *\xycircle<3pc,3pc>{};
        \POS(3.4,1.8)  *@{*}*+!D{u_3}="c";
         \POS (3.01, .9) *@{*}*+!R{u_4}="g";
         \POS (4.3, 1.95) *@{*}*+!D{u_2};
         \POS(5.0,.9)  *@{*}*+!L{u_1}="h";
              \POS(3.8, 1.98) \ar@{->}(3.8, 1.98);(3.75, 1.97)
        \POS(3.07, 1.35) \ar@{->}(3.07, 1.35);(3.04, 1.298)
          \POS(4.8,1.57)   \ar@{->}(4.8,1.57);(4.9,1.44);
        \POS(4.94,.6)   \ar@{->}(4.94,.6);(4.93,.585);
                \POS(4.3,.06)   \ar@{->}(4.3,.06);(4.24,.038);
        \POS "h" \ar@{-} "g";
             \POS(3.8,.9) \ar@{->}(3.8,.9);(3.76,.9)
        \POS (4, -0.25) *@{}*+!D{\cdots};
        \POS(3.2,.4)    \ar@{->}(3.2,.4) ;(3.19,.415) ;
        \POS (3.4, .2) *@{*}*+!R{u_5};
        \POS (4.6, .2) *@{*}*+!L{u_n};
 \endxy
$
\vskip0.00001mm
\hskip1.7cm $\theta(1,1,n-4)$ \hskip3.4cm $\widehat{\theta}$ \hskip3.7cm  $\theta(0,2,n-4)$
\vskip0.1mm
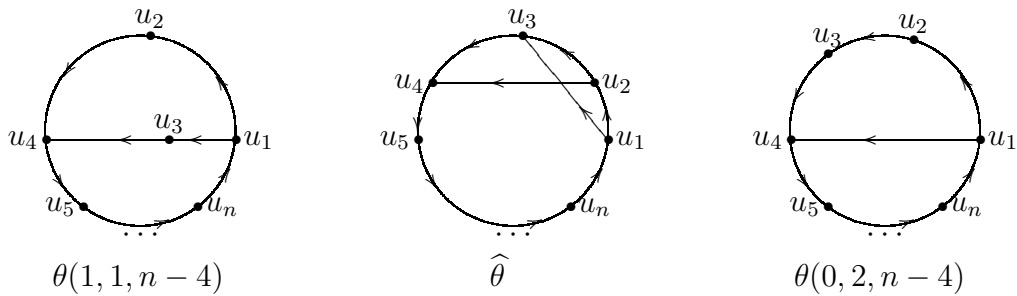
\hskip2.5cm  Fig.5. \hskip.1cm  The digraphs $\theta(1,1,n-4)$, $\widehat{\theta}$,  and $\theta(0,2,n-4)$.

\begin{theo}\label{theo54}
Let  $n\geq 4$. Then $\theta(0,1,n-3)$, $\theta(1,1,n-4)$, $\theta(0,2,n-4)$ are the digraphs which achieve the second,
the third and the fourth minimum signless Laplacian spectral radius among all strongly connected digraphs on $n$ vertices,
respectively.
\end{theo}

\begin{proof}
By $\mathcal{G}_{n, n}=\{\overrightarrow{C_n}\}$ and Theorems \ref{theo52}$\sim$\ref{theo53}, it is clear that $C_{n,n-1}\cong \theta(0,1,n-3)$ is the unique digraph with the second minimum signless Laplacian spectral radius.

Note that $\mathcal{G}_{n, n-1}=\{\theta(0,1,n-3), \theta(1,1,n-4), \widehat{\theta}\}$ and
$2=q(\overrightarrow{C_n})<q(\theta(0,1,n-3))<\min \{q(\theta(1,1,n-4)), q(\widehat{\theta}), q(\theta(0,2,n-4))\}$,
we only need to show that $q(\theta(1,1,n-4))<q(\theta(0,2,n-4))<q(\widehat{\theta})$ by Theorems \ref{theo52}$\sim$\ref{theo53}.

Let $P_{\theta(1,1,n-4)}(x)$, $P_{\theta(0,2,n-4)}(x)$,  $P_{\widehat{\theta}}(x)$ be the signless Laplacian characteristic polynomial of
$\theta(1,1,n-4), \theta(0,2,n-4)$ and $\widehat{\theta}$, respectively. By directly calculating, we have

\vskip2mm
\noindent
$P_{\theta(1,1,n-4)}(x)=\left|\begin{array}{cccccccc}
            x-2 & -1 & -1 & 0 & 0 &  0 & \cdots & 0 \\
            0 & x-1 & 0 & -1 & 0  &  0 & \cdots & 0 \\
            0 & 0 & x-1 & -1 & 0 &  0 & \cdots & 0 \\
            0 & 0 & 0 & x-1 & -1 & 0 & \cdots &0 \\
            0 & 0 & 0 & 0 & x-1 & -1 & \cdots &0 \\
            \vdots & \vdots & \vdots &\vdots  &\ddots  & \ddots & \ddots & \vdots  \\
            0 & 0 & 0 & 0 &\cdots & 0 & x-1 & -1  \\
            -1 & 0 & 0 & 0 & 0 & \cdots & 0 & x-1
          \end{array}\right|
$
\vskip2mm
\hskip1.8cm  $=(x-1)[(x-2)(x-1)^{n-2}-2]$,

\vskip2.4mm

\noindent$P_{\theta(0,2,n-4)}(x)=\left|\begin{array}{cccccccc}
            x-2 & -1 & 0 & -1 & 0 &  0 & \cdots & 0 \\
            0 & x-1 & -1 & 0 & 0  &  0 & \cdots & 0 \\
            0 & 0 & x-1 & -1 & 0 &  0 & \cdots & 0 \\
            0 & 0 & 0 & x-1 & -1 & 0 & \cdots &0 \\
            0 & 0 & 0 & 0 & x-1 & -1 & \cdots &0 \\
            \vdots & \vdots & \vdots &\vdots  &\ddots  & \ddots & \ddots & \vdots  \\
            0 & 0 & 0 & 0 &\cdots & 0 & x-1 & -1  \\
            -1 & 0 & 0 & 0 & 0 & \cdots & 0 & x-1
          \end{array}\right|
$
\vskip2mm
\hskip1.8cm  $=(x-1)^2[(x-2)(x-1)^{n-3}-1]-1$,

\vskip2.4mm
\noindent$P_{\widehat{\theta}}(x)=\left|\begin{array}{cccccccc}
            x-2 & -1 & -1 & 0 & 0 &  0 & \cdots & 0 \\
            0 & x-2 & -1 & -1 & 0  &  0 & \cdots & 0 \\
            0 & 0 & x-1 & -1 & 0 &  0 & \cdots & 0 \\
            0 & 0 & 0 & x-1 & -1 & 0 & \cdots &0 \\
            0 & 0 & 0 & 0 & x-1 & -1 & \cdots &0 \\
            \vdots & \vdots & \vdots &\vdots  &\ddots  & \ddots & \ddots & \vdots  \\
            0 & 0 & 0 & 0 &\cdots & 0 & x-1 & -1  \\
            -1 & 0 & 0 & 0 & 0 & \cdots & 0 & x-1
          \end{array}\right|
$
\vskip2mm
\hskip0.5cm  $=(x-1)[(x-2)^2(x-1)^{n-3}-2]$.

When $x>2$, we note that

\hskip3.5cm $P_{\theta(1,1,n-4)}(x)-P_{\widehat{\theta}}(x)=(x-1)^{n-2}(x-2)>0,$

\hskip3cm$ P_{\theta(0,2,n-4)}(x)-P_{\widehat{\theta}}(x)=(x-2)[(x-1)^{n-2}-(x-2)]>0,$

\hskip3.7cm $P_{\theta(1,1,n-4)}(x)-P_{\theta(0,2,n-4)}(x)=(x-2)^2>0.$

\noindent Then $q(\theta(1,1,n-4))<q(\theta(0,2,n-4))<q(\widehat{\theta})$.
\end{proof}

\section{The maximum signless Laplacian spectral radius of strongly connected digraph with given vertex connectivity}

\hskip.6cm
In this section, we will discuss the maximum signless Laplacian spectral radius of strongly connected digraph with given vertex connectivity, and propose some open problem.

Let $\mathcal{D}_{n, k}$ denote the set of strongly connected digraphs on $n$ vertices with vertex connectivity $\kappa(D)=k$. If $k=n-1$, then $\mathcal{D}_{n, n-1}=\{\overset{\longleftrightarrow}{K_n}\}$. 
Now we only need to discuss the cases $1\leq k\leq n-2$.

\vskip0.2cm
$
\hskip6cm
 \xy 0;/r3pc/: \POS (4,3) *\xycircle<3pc,1pc>{};  \POS (4,2) *\xycircle<3pc,1pc>{};  \POS (4,1) *\xycircle<3pc,1pc>{};
         \POS (3.5, 1) *@{*}="a"; \POS(4.5,1)  *@{*}="b";\POS(5.8,0.7) *@{}*+!D{\overset{\longleftrightarrow\quad\quad}{K_{n-m-k}}};
         \POS (3.5, 2) *@{*}="c"; \POS(4.5,2)  *@{*}="d";\POS(5.5,1.7) *@{}*+!D{\overset{\longleftrightarrow}{K_k}};
         \POS (3.5, 3) *@{*}="e"; \POS(4.5,3)  *@{*}="f"\POS(5.5,2.7) *@{}*+!D{\overset{\longleftrightarrow}{K_m}};
         \POS(4.28,1)  *@{}*+!R{\cdots};\POS(4.28,2)  *@{}*+!R{\cdots}; \POS(4.28,3)  *@{}*+!R{\cdots};
         \ar@{-}"a";"c"; \ar@{-}"a";"d"; \ar@{-}"b";"c"; \ar@{-}"b";"d"; \ar@{-}"c";"e"; \ar@{-}"c";"f"; \ar@{-}"d";"e"; \ar@{-}"d";"f";
         \ar@{-}\ar@/_{1.7pc}/ "a";"e";
         \ar@{-}\ar@/_{-1.7pc}/ "b";"f";
 \endxy
 $
\vskip0.00001mm
\hskip5cm  Fig.6. \hskip.1cm  The digraph $\overrightarrow{K}(n, k, m)$.
\vskip0.2cm

Let $D_1\bigtriangledown D_2$ denote the digraph obtained from two disjoint digraphs $D_1$, $D_2$ with vertex set $V=V(D_1)\cup V(D_2)$ and
arc set $E=E(D_1)\cup E(D_2)\cup \{(u,v), (v,u) | u\in V(D_1), v\in V(D_2)\}$.
Let $1\leq k\leq n-2$, $1\leq m\leq n-k-1$, and  $\overrightarrow{K}(n, k, m)$  denote the digraph $\overset{\longleftrightarrow}{K_k}\bigtriangledown(\overset{\longleftrightarrow}{K_m}\cup \overset{\longleftrightarrow\quad\quad}{K_{n-m-k}})\cup E_1$,
where $E_1=\{(u,v) |u\in V(\overset{\longleftrightarrow}{K_m}), v\in V(\overset{\longleftrightarrow\quad\quad}{K_{n-m-k}})\}$
(see Fig.6). Clearly, $\overrightarrow{K}(n, k, m)\in \mathcal{D}_{n, k}$.

Let  $\overrightarrow{\mathcal{K}}(n, k)=\{\overrightarrow{K}(n, k, m)|1\leq m\leq n-k-1\}$ where $1\leq k\leq n-2$.
Clearly, $\overrightarrow{\mathcal{K}}(n, k)\subseteq\mathcal{D}_{n, k}$.

\begin{prop}\label{prop61}{\rm(\cite{1976})}
Let $D$ be a  strongly connected digraph with $\kappa(D)=k$. Suppose that $S$ is a $k$--vertex cut of $D$ and $D_1$, $D_2$, \ldots, $D_t$ are the strongly connected components of $D-S$. Then there exists an ordering of $D_1$, $D_2$, \ldots, $D_t$ such that for $1\leq i\leq t$ and $v\in V(D_i)$, every tail of $v$ is in $\bigcup\limits^{i-1}_{j=1} D_j$.
\end{prop}

\begin{rem}\label{rem62}
By Proposition \ref{prop61}, we know that $D_1$ with $|V(D_1)|=m$ is the strongly connected component of $D-S$ where the inneighbors of vertices of $V(D_1)$ in $D-S$ are zero. Let $D_2=D-S-D_1$. We add arcs to $D$ until both induced subdigraph of $V(D_1)\cup S$ and induced subdigraph of $V(D_2)\cup S$ attain to complete digraphs, add arc $(u,v)$ for any $u\in V(D_1)$ and any $v\in V(D_2)$. Denote the new digraph by $H$. Since $D$ is $k$--strongly connected, then $H=\overrightarrow{K}(n,k,m)\in \overrightarrow{\mathcal{K}}(n, k)\subseteq \mathcal{D}_{n, k}$.
Since $D$ is the subdigraph of $H$, then $q(D)\leq q(H)$, with equality if and only if $D\cong H$ by Corollary \ref{cor25}. Therefore, the digraph which achieves the maximum signless Laplacian spectral radius in $\mathcal{D}_{n, k}$ must be some digraph in $\overrightarrow{\mathcal{K}}(n, k)$.
\end{rem}

\begin{theo}\label{theo63}
Let $n,k,m$ be positive integers with $1\leq k\leq n-2$ and $1\leq m\leq n-k-1$. 
Then $q(\overrightarrow{K}(n, k, m))=\frac{3n-m-4+\sqrt{(n-3m)^2+8mk}}{2}$.
\end{theo}
\begin{proof}
Let $D=\overrightarrow{K}(n, k, m),$  $S$ be a $k$--vertex cut of $D$. Suppose that $D_1$ with $|V(D_1)|=m$ and $D_2$ with $|V(D_2)|=n-m-k=t$ are two strongly connected components, i.e. two complete subdigraphs of $D-S$ with arcs $E_1=\{(u, v) | u\in V(D_1), v\in V(D_2)\}$. Then

$Q(D)=\begin{bmatrix}
J_{m\times m}+(n-2)I_{m\times m} & J_{m \times k} & J_{m\times t}\\
J_{k\times m} & J_{k \times k}+(n-2)I_{k\times k} & J_{k\times t}\\
O_{t\times m} & J_{t\times k} & J_{t\times t}+(k+t-2)I_{t\times t}
\end{bmatrix}$
\vskip.2cm
\noindent and the signless Laplacian characteristic polynomial of $D$ is

$P_{D}(x)=|xI_{n\times n}-Q(D)|$

\hskip1.2cm$=\begin{vmatrix}
(x-n+2)I_{m\times m}-J_{m\times m} & -J_{m \times k} & -J_{m\times t}\\
-J_{k\times m} & (x-n+2)I_{k\times k}-J_{k \times k} & -J_{k\times t}\\
O_{t\times m} & -J_{t\times k} & (x-t-k+2)I_{t\times t}-J_{t\times t}
\end{vmatrix}$

\vskip.2cm
\hskip1.2cm$=(x-n+2)^{m+k-1}(x-k-t+2)^{t-1}(x-n-m+2)\left[x-2(k+t-1)-\frac{2mk}{x-n-m+2}\right]$
\vskip.2cm
\hskip1.2cm$=(x-n+2)^{m+k-1}(x-n+m+2)^{t-1}[x^2-(3n-m-4)x+2(n-m-1)(n+m-2)-2mk]$.

Let $D'$ be the proper subdigraph of $D$ and
\vskip.2cm
$Q(D')=\begin{bmatrix}
J_{(m+k)\times (m+k)}+(n-2)I_{(m+k)\times (m+k)} & O_{(m+k)\times t}\\
O_{t\times (m+k)} &  J_{t\times t}+(k+t-2)I_{t\times t}
\end{bmatrix}$,
\vskip.2cm
\noindent
then $q(D)>q(D')=\max\{n+m+k-2, 2n-2m-k-2\}$ by $\rho(aJ_{n\times n}+bI_{n\times n})=na+b$.

Note that $\max\{n+m+k-2, 2n-2m-k-2\}>\max\{n-2,n-m-2\}=n-2$, then $q(D)$ is equal to the solution of the the quadratic equation $x^2-(3n-m-4)x+2(n-m-1)(n+m-2)-2mk=0$, thus $q(D)=\frac{3n-m-4+\sqrt{(n-3m)^2+8mk}}{2}$ or $q(D)=\frac{3n-m-4-\sqrt{(n-3m)^2+8mk}}{2}$.


\vskip.2cm
If $q(D)=\frac{3n-m-4-\sqrt{(n-3m)^2+8mk}}{2}$, then

\hskip1cm $q(D')<q(D)<\frac{3n-m-4-|n-3m|}{2}=\left\{\begin{array}{ll}
n+m-2,& \mbox{if } n\geq 3m;\\
2n-2m-2, & \mbox{if } 0<n<3m.
\end{array}
\right.$


Now we will show there exists a condiction since $n+m+k-2\leq q(D')<q(D)$.
When $n\geq 3m$, there is a contradiction by $n+m+k-2<q(D)<n+m-2$; when $0<n<3m$, there is also a contradiction by $n+m+k-2<q(D)<2n-2m-2$.
Combining the above arguments, we have $q(D)=\frac{3n-m-4+\sqrt{(n-3m)^2+8mk}}{2}$.
\end{proof}

\begin{rem}\label{rem65}
Note that $\overset{\longleftrightarrow}{K_n}$ is the unique digraph which achieves the  maximum signless Laplacian spectral radius $2n-2$  among all strongly connected digraphs, and $\overrightarrow{K}(n, n-2, 1)\cong \overset{\longleftrightarrow}{K_n}-\{(u, v)\}$ where $u,v\in V(\overset{\longleftrightarrow}{K_n})$,
by Lemma \ref{lem23} and Theorem \ref{theo63}, we are sure that $\overrightarrow{K}(n, n-2, 1)$ is the unique digraph which achieves the second maximum signless Laplacian spectral radius $\frac{3n-5+\sqrt{n^2+2n-7}}{2}$  among all strongly connected digraphs.
\end{rem}

\begin{theo}\label{theo66}
Let $D$ be a strongly connected digraph, $D\not\cong \overset{\longleftrightarrow}{K_n},$ and $D\not\cong \overrightarrow{K}(n, n-2, 1)$.
Then $q(D)< \frac{3n-5+\sqrt{n^2+2n-7}}{2}$.
\end{theo}

Since $\mathcal{D}_{n, n-1}=\{\overset{\longleftrightarrow}{K_n}\}$, we know $\overrightarrow{K}(n, n-2, 1)$
is the unique digraph which achieves the maximum signless Laplacian spectral radius in $\mathcal{D}_{n, n-2}$ by Remark \ref{rem65} or Theorem \ref{theo66}.
Thus we can proposed  the following conjecture and we have seen that when $k=n-2$, the conjecture is true.

\begin{con}\label{con67}
Let $n,k$ be given positive integers with $1\leq k\leq n-2$,  $D\in\mathcal{D}_{n, k}$.
Then $q(D)\leq \frac{3n-5+\sqrt{(n-3)^2+8k}}{2}$ with equality if and only if  $G\cong \overrightarrow{K}(n, k, 1)$.
\end{con}

\section{Some notes on the spectral radius of strongly connected digraphs}

\hskip.6cm
In Section 3 $\sim$ Section 6, we use some similar technique which applied in \cite{2012DM}.
Although there are some defects in Section 2 $\sim$ Section 3 in  \cite{2012DM} which can be revised by similar proofs of this paper,
the results are well done and  there are some useful techniques which can be used to study  strongly connected digraphs.
Now in this section, we only show that there are more results can be obtained on  the spectral radius of strongly connected digraphs.

Note that $\overrightarrow{C_n}$ is the unique digraph with the minimum signless Laplacian spectral radius among all strongly connected digraphs on $n$ vertices. In \cite{2012LAA}, the authors  characterized the extremal digraphs which achieve the maximum and minimum spectral radius among all strongly connected bicyclic digraphs, and they proposed the following open problem.

\begin{prob}\label{prob71}
Is the digraph $\theta (0, 1, n-3)$ achieving the second minimum spectral radius among all strongly connected digraphs?
\end{prob}

We confirm Problem \ref{prob71} 
and we can obtain more by the similar proofs of Theorems \ref{theo53}$\sim$\ref{theo54}.

\begin{theo}\label{theo72}
Let $n\geq 4$. Then  $\rho(\overrightarrow{C_n})<\rho(C_{n,n-1})<\rho(C_{n,n-2})<\cdots<\rho(C_{n,2})$.
\end{theo}

 \begin{theo}\label{theo73}
Let  $n\geq 4$. Then $\theta(0,1,n-3)$, $\theta(1,1,n-4)$, $\theta(0,2,n-4)$ are the digraphs which achieve the second,
the third and the fourth minimum  spectral radius among all strongly connected digraphs on $n$ vertices, respectively.
\end{theo}

Note that $\overset{\longleftrightarrow}{K_n}$ is the unique digraph which achieves the  maximum  spectral radius $n-1$  among all strongly connected digraphs,
we have the following result by  Theorem 4.2 in \cite{2012DM} and Lemma \ref{lem23}.
\begin{theo}\label{theo74}
Let  $n\geq 4$. Then $\overrightarrow{K}(n, n-2, 1)$ is the unique digraph which achieves the second maximum  spectral radius $\frac{n-2+\sqrt{n^2-4}}{2}$  among all strongly connected digraphs.
\end{theo}


\end{document}